\documentclass{article}
\usepackage[utf8]{inputenc}

\title{An inequality for the heat kernel on an Abelian Cayley graph}
\author{Thomas McMurray Price
\\ \href{mailto:tom.price.math@gmail.com}{tom.price.math@gmail.com} }

\usepackage{amsthm}
\usepackage{amssymb}
\usepackage{amsmath}
\usepackage{hyperref}

\hypersetup{
  colorlinks   = true, %Colours links instead of ugly boxes
  urlcolor     = blue, %Colour for external hyperlinks
  linkcolor    = blue, %Colour of internal links
  citecolor   = red %Colour of citations
}

\newtheorem{theorem}{Theorem}[section]
\newtheorem{corollary}[theorem]{Corollary}
\newtheorem{lemma}[theorem]{Lemma}

\newtheorem{proposition}[theorem]{Proposition}

\theoremstyle{definition}
\newtheorem{definition}[theorem]{Definition}

\theoremstyle{remark}
\newtheorem{remark}[theorem]{Remark}

\DeclareMathOperator{\cexp}{exp^*}
\DeclareMathOperator{\arccosh}{arccosh}

\begin{document}

\maketitle

\begin{abstract}

We demonstrate a relationship between the heat kernel on a finite weighted Abelian Cayley graph and Gaussian functions on lattices. This can be used to prove a new inequality for the heat kernel on such a graph: when $t \leq t'$,
$$\frac{H_t(u, v)}{H_t(u,u)} \leq \frac{H_{t'}(u, v)}{H_{t'}(u,u)}$$
This was an open problem posed by Regev and Shinkar.

\end{abstract}

\section{Introduction}

Let $\Gamma = (V, E)$ be a finite undirected weighted graph. Let $L$ be the graph laplacian of $\Gamma$. We then define the \textbf{heat kernel} $H_t$ as the matrix exponential $e^{-tL}$, for any $t \in \mathbb{R}_{> 0}$. The value $H_t(u, v)$ can be thought of as the probability that a continuous-time random walk starting at $u$ ends up at $v$ after time $t$. See Section 1.1 of \cite{RS} for more detail on this interpretation. We will say that $G$ has \textbf{monotonic diffusion} if the following holds whenever $t' \geq t$, for all vertices $v$ and $u$:

$$\frac{H_t(u, v)}{H_t(u,u)} \leq \frac{H_{t'}(u, v)}{H_{t'}(u,u)}$$

We'll say that $\Gamma$ is a \textbf{weighted Abelian Cayley graph} if there is an Abelian group structure on $V$ under which $\Gamma$ is translation-invariant. In this note we will prove the following theorem:

\begin{theorem} \label{thm:main_ineq}
Any weighted Abelian Cayley graph has monotonic diffusion.
\end{theorem}

For arbitrary weighted undirected graphs, monotonic diffusion often fails to hold. See the appendix of \cite{RS} for a simple case. In 2013, Peres asked whether vertex-transitive graphs necessarily have monotonic diffusion. This was resolved in the negative by Regev and Shinkar in \cite{RS}. It also makes sense to ask whether a Riemannian manifold has monotonic diffusion, since there is a uniquely determined heat kernel in this setting, provided the manifold is complete and the Ricci curvature is bounded from below \cite{Ch}. We present an argument in Appendix \ref{app:monotonicity_counterexample}, due to Jeff Cheeger, that monotonicity does not always hold for Riemannian manifolds. In the other direction, it was shown in \cite{RSD} that all flat tori have monotonic diffusion. Since weighted Abelian Cayley graphs are the closest thing we have to flat tori in the world of graph theory, it is natural to ask whether they have monotonic diffusion. This problem was posed in \cite{RS} and is resolved, in the affirmative, in this note.

We will use the following strategy: we show that, if $\upsilon$ is a nonnegative, even, real-valued function on a finite Abelian group, then the convolutional exponential $\cexp(\upsilon)$ can be represented as a pushforward of a Gaussian function on a lattice. This will allow us to apply an inequality in \cite{RSD} regarding Gaussian sums on lattice cosets.

\section{Gaussian pushforwards}

Throughout this paper, $G$ will refer to an arbitrary finite Abelian group.

\begin{definition}

A \textbf{lattice} is a discrete subgroup of a finite-dimensional inner product space.  If $L$ is a lattice, we define the function $\rho: L \to \mathbb{R}$ as $\rho(x) = e^{-\pi \langle x , x \rangle}$. When $S \subseteq L$, we'll use the notation $\rho(S)$ to mean $\sum_{x \in S} \rho(x)$.

\end{definition}

\begin{definition}

We'll say that a function $\chi: G \to \mathbb{R}$ is a \textbf{Gaussian pushforward} if there exists a lattice $L$ and a group homomorphism $h: L \to G$ such that $\chi(g) = \rho(h^{-1}(g))$ for all $g \in G$. We'll use $X$ to denote the set of all Gaussian pushforwards on $G$.  

\end{definition}

\begin{proposition} \label{thm:X_closed}

$X$ is closed under convolution.

\end{proposition}

\begin{proof}

Suppose that $\chi_1 \in X$ and $\chi_2 \in X$. Then we have lattices $L_1$ and $L_2$, as well as homomorphisms $h_1: L_1 \to G$ and $h_2: L_2 \to G$, such that $\chi_1(g) = \rho(h_1^{-1}(g))$ and $\chi_2(g) = \rho(h_2^{-1}(g))$ for all $g \in G$. Take $L_3$ to be the orthogonal direct sum of $L_1$ and $L_2$. We have the homomorphism $h_3: L_3 \to G$ given by $h_3((x_1, x_2)) = h_1(x_1) + h_2(x_2)$. Define $\chi_3: G \to \mathbb{R}$ as $\chi_3(g) = \rho(h_3^{-1}(g))$. Then $\chi_3$ = $\chi_1 * \chi_2$. We also clearly have that $\chi_3 \in X$. This yields the result.

\end{proof}

\begin{remark}
Curiously, we also have that $X$ is closed under multiplication, if we instead take $L_3$ to be the sublattice of $L_1 \oplus L_2$ of points $(x_1, x_2)$ with $h_1(x_1) = h_2(x_2)$, and $h_3((x_1, x_2))$ to be $h_1(x_1)$.
\end{remark}

\begin{proposition} \label{thm:chi_inequality_rollary}
If $\chi \in X$, then for all $g_1, g_2 \in G$, we have

$$\frac{\chi(g_1)\chi(g_2)}{\chi(0)} \leq \frac{\chi(g_1 + g_2) + \chi(g_1 - g_2)}{2}$$
\end{proposition}

\begin{proof}

This follows immediately from (4c) of corollary 2.2 of \cite{RSD}, if we take the lattice $\mathcal{L}$ to be the kernel of a homomorphism from a lattice to $G$ whose corresponding Gaussian pushforward function is $\chi$.

\end{proof}

From continuity, we immediately have the following stronger statement:

\begin{corollary} \label{thm:chi_inequality}
If $\chi$ is in the topological closure $\bar{X}$ of $X$ then for all $g_1, g_2 \in G$, we have

$$\frac{\chi(g_1)\chi(g_2)}{\chi(0)} \leq \frac{\chi(g_1 + g_2) + \chi(g_1 - g_2)}{2}$$
\end{corollary}

\begin{proposition} \label{thm:chi_convolve_even}

If $\chi$ is in $\bar{X}$, $\upsilon$ is a nonnegative, even, real-valued function on $G$, and $\omega = \chi * \upsilon$, then for all $g \in G$,

$$\frac{\chi(g)}{\chi(0)} \leq \frac{\omega(g)}{\omega(0)}$$

\end{proposition}

\begin{proof}

For all $g \in G$, we have

$$\omega(g) = \sum_{g' \in G} \chi(g - g')\upsilon(g')$$
$$= \sum_{g' \in G} \chi(g + g')\upsilon(-g')$$
$$= \sum_{g' \in G} \chi(g + g')\upsilon(g')$$

Taking the mean of the first and third sums above, we have:

$$\omega(g) = \sum_{g' \in G} \frac{\chi(g + g') + \chi(g - g')}{2}\upsilon(g')$$

Applying \ref{thm:chi_inequality},

$$\omega(g) \geq \sum_{g' \in G} \frac{\chi(g)\chi(g')\upsilon(g')}{\chi(0)}$$
$$= \frac{\chi(g)}{\chi(0)}\sum_{g' \in G}\chi(g')\upsilon(-g')$$
$$= \frac{\chi(g)\omega(0)}{\chi(0)}$$

Dividing by $\omega(0)$ yields the result.

\end{proof}

\section{The convolutional exponential of nonnegative even functions}

\begin{definition}
	Suppose $\upsilon$ is a real-valued function on $G$. We define $\cexp(\upsilon)$, the \textbf{convolutional exponential} of $\upsilon$, by
	$$\cexp(\upsilon) = \sum\limits_{n=0}^{\infty}\frac{\upsilon^{*n}}{n!},$$
	where $\upsilon^{*n}$ is the $n$th convolutional power of $\upsilon$. From the convolution theorem, we can equivalently define $\cexp(\upsilon)$ by the equation
	$$\widehat{\cexp(\upsilon)} = \exp(\hat{\upsilon}),$$
	where $\exp$ is the pointwise exponential operator. From this perspective, it's clear that $\cexp(a + b) = \cexp(a)*\cexp(b)$.
\end{definition}

In this section, we'll show that the convolutional exponential of any nonnegative even function on $G$ is a Gaussian pushforward.

\begin{definition}
We'll use $\Phi$ to refer to the set of functions on $G$ of the form $\phi(g) = \delta(g - g_0) + \delta(g + g_0)$ for some $g_0 \in G$. Here we have $\delta(0) = 1$ and $\delta(x) = 0$ for $x \neq 0$. We clearly have that $\Phi$ is a basis for the even functions on $G$.
\end{definition}

\begin{definition}

We'll say that a sequence of real-valued functions $\upsilon_n$ on $G$ is $O(f(n))$ if, for any $g \in G$, we have that $\upsilon_n(g)$ is $O(f(n))$. Since all norms on a finite-dimensional vector space are equivalent, this is the same as saying that $\lVert \upsilon_n \rVert$ is $O(f(n))$, for any choice of norm on $\mathbb{R}^G$. We also have that $\upsilon_n$ is $O(f(n))$ iff $\hat{\upsilon}_n$ is $O(f(n))$, by the Plancherel theorem.

\end{definition}

\begin{lemma} \label{thm:reals_2}
    If $a$ and $b$ are both in $[0, C]$, then for all $n \in \mathbb{N}$, we have $|a^n - b^n| \leq nC^{n- 1}|a - b|$.
\end{lemma}

\begin{proof}
    Without loss of generality, assume $a \geq b$. We have:
    $$a^n - b^n = (a - b)\sum_{i = 0}^{n - 1}a^ib^{n - 1 - i}$$
    Since each term $a^ib^{n - 1 - i}$ of the summation is at most $C^{n - 1}$, we have:
    $$a^n - b^n \leq nC^{n- 1}(a - b)$$
    Since both sides are nonnegative, the above inequality proves the lemma.
\end{proof}

\begin{lemma} \label{thm:reals_3}
    If $a_n$ and $b_n$ are both $1 + O(1/n)$, and $a_n - b_n$ is $O(f(n))$, then $a_n^n - b_n^n$ is $O(nf(n))$.
\end{lemma}

\begin{proof}
    We have, for some $C$, that $a_n \leq 1 + C/n$ and $b_n \leq 1 + C/n$. By applying \ref{thm:reals_2}, we have
    $$|a_n^n - b_n^n| \leq n(1 + C/n)^{n-1}|a_n - b_n| \leq ne^C|a_n - b_n| \in O(nf(n))$$
\end{proof}

\begin{lemma} \label{thm:gaussian_approx}
Suppose $\phi \in \Phi$ and $\alpha \in \mathbb{R}$. Then we can find a sequence of Gaussian pushforwards $(\chi_n)$ such that $\delta + \alpha\phi/n - \chi_n$ is $O(1/n^4)$.
\end{lemma}

\begin{proof}
We'll choose $\chi_n$ arbitrarily when $n \leq \alpha$, and the rest of this proof will be concerned with the tail $(\chi_n)_{n > \alpha}$. Choose a $g_0$ such that $\phi(g) = \delta(g + g_0) + \delta(g - g_0)$ for all $g \in G$. Let $r_n = \sqrt{\ln(n/\alpha)/\pi}$, so that $\alpha/n = \rho(r_n)$. Let $L_n$ be the lattice in $\mathbb{R}$ of reals of the form $kr_n$, with $k \in \mathbb{Z}$. We then have a unique group homomorphism $h_n$ from $L_n$ to $G$ that sends $r_n$ to $g_0$. We define $\chi_n$ by $\chi_n(g) = \rho(h_n^{-1}(g))$. Let $A_n = \{ r_n, 0, -r_n \}$ and $B_n = L_n \setminus A_n$. We can break $\chi_n$ into two smaller sums:

$$\chi_n(g) = \rho(h_n^{-1}(g) \cap A_n) + \rho(h_n^{-1}(g) \cap B_n)$$

The left term is equal to $\delta + \alpha\phi/n$. The right term is $O(1/n^4)$. This yields the result.

\end{proof}

\begin{lemma} \label{thm:cexp_gaussian_1}
If $\phi \in \Phi$ and $\alpha \in \mathbb{R}_{>0}$, then $\cexp(\alpha \phi)$ is in the topological closure $\bar{X}$ of $X$.
\end{lemma}

\begin{proof}

Let $\psi = \cexp(\alpha \phi)$. From the convolution theorem, we have that $\hat{\psi} = \exp(\alpha\hat{\phi})$. Therefore, we have
$$\hat{\psi} = \lim_{n \to \infty} (1 + \frac{\alpha\hat{\phi}}{n})^n$$.

From \ref{thm:gaussian_approx}, we have a sequence of Gaussian pushforwards $(\chi_n)$ such that ${\delta + \alpha \phi/n - \chi_n}$ is $O(1/n^4)$. Taking the Fourier transform, we have that ${1 + \alpha\hat{\phi}/n - \hat{\chi}}$ is $O(1/n^4)$. We then have from \ref{thm:reals_3} that $(1 + \alpha\hat{\phi}/n)^n - \hat{\chi_n}^n$ is $O(1/n^3)$. Combining this with the formula for $\hat{\psi}$ above, we have:
$$\lim_{n \to \infty} \hat{\chi_n}^n = \hat{\psi}$$
Applying the convolution theorem then yields:
$$\lim_{n \to \infty} \chi_n^{*n} = \psi$$.

From \ref{thm:X_closed}, we know that $\chi_n^{*n}$ is a Gaussian pushforward, so this last limit proves the result.

\end{proof}

\begin{theorem} \label{thm:cexp_gaussian_2}
Suppose $\upsilon$ is a nonnegative even function on $G$. Then $\cexp{\upsilon}$ is in $\bar{X}$.
\end{theorem}

\begin{proof}
We can clearly represent $\upsilon$ as a sum $\alpha_1 \phi_1 + \cdots + \alpha_n \phi_n$, with $\alpha_i$ being positive scalars and $\phi_i \in \Phi$. Then, applying the convolution theorem, we can represent $\cexp(\upsilon)$ as the convolution $\cexp(\alpha_1 \phi_1) * \cdots *\cexp(\alpha_n \phi_n)$. From \ref{thm:cexp_gaussian_1}, we know that each individual term in this convolution is in the closure of $X$, so the result follows from \ref{thm:X_closed} and the continuity of the convolution operation.
\end{proof}

\section{The heat kernel on weighted Abelian Cayley graphs}

In this section, we prove \ref{thm:main_ineq}. The hard work has already been done in \cite{RSD} and in the previous section, now we just need to translate these results into the language of graph theory.

\begin{proof} [Proof of Theorem \ref{thm:main_ineq}]

We define $\tau$ to be the element of $\mathbb{R}^V$ obtained by applying the graph laplacian to $-\delta$. In other words, we have $\tau: V \to \mathbb{R}$ and $\tau(v) = -L(0, v)$, with $L$ the graph laplacian. Then $\tau$ is an even function and is nonnegative everywhere except at $0$. We also have that the linear operator determined by the graph laplacian is just convolution with $-\tau$; it follows that the heat kernel matrix $H_t$ corresponds to convolution with $\cexp(t \tau)$, and therefore $H_t(u, v) = \cexp(t \tau)(u - v)$. So, to prove \ref{thm:main_ineq}, it suffices to show that, for $t \leq t'$, and for any $v \in V$,

$$\frac{\cexp(t\tau)(v)}{\cexp(t\tau)(0)} \leq \frac{\cexp(t'\tau)(v)}{\cexp(t'\tau)(0)}$$

Let $\tau': V \to \mathbb{R}$ be given by $\tau'(0) = 0$ and $\tau'(v) = \tau(v)$ for $v \neq 0$. Then $\tau'$ is a nonnegative even function. We also have that $\cexp(t\tau') = e^{-\tau(0)}\cexp(t\tau)$. It therefore suffices to show that, for $t \leq t'$, and for any $v \in V$,

$$\frac{\cexp(t\tau')(v)}{\cexp(t\tau')(0)} \leq \frac{\cexp(t'\tau')(v)}{\cexp(t'\tau')(0)}$$

Since we have $\cexp(t'\tau') = \cexp(t\tau') * \cexp((t' - t)\tau')$, the result follows immediately from \ref{thm:cexp_gaussian_2} and \ref{thm:chi_convolve_even}.

\end{proof}

\section{Directions for further work}

In order to prove \ref{thm:main_ineq}, we needed the inequality \ref{thm:chi_inequality} for $\chi$ a Gaussian pushforward:

$$\frac{\chi(g_1)\chi(g_2)}{\chi(0)} \leq \frac{\chi(g_1 + g_2) + \chi(g_1 - g_2)}{2}$$

which itself follows from the stronger inequality, proven in \cite{RSD}:

\begin{equation} \label{eqn:RSD_main}
\chi(g_1)^2\chi(g_2)^2 \leq \chi(g_1+g_2)\chi(g_1 - g_2)\chi(0)^2
\end{equation}

The proof of monotonic diffusion on flat tori also uses \eqref{eqn:RSD_main}, except $\chi(g)$ is replaced by $H_t(0, g)$, with $H$ the heat kernel on the torus in question.

So, if we want to see where else we have monotonic diffusion, it is natural to ask: where else does \eqref{eqn:RSD_main} hold? On the surface, it looks we need group structure to even state \eqref{eqn:RSD_main}. However, \eqref{eqn:RSD_main} can be reformulated in the language of Riemannian symmetric spaces. Indeed, suppose $M$ is a Riemannian symmetric space, that is, $M$ is a connected Riemannian manifold equipped with an isometric automorphism $s_p$ for each point $p$, such that $s_p(p) = p$ and the derivative of $s_p$ at $p$ is the negation map on $T_p$. For flat tori then, we have a symmetric space structure with $s_x(y) = 2x - y$, and so the following is equivalent to \eqref{eqn:RSD_main}:

\begin{equation} \label{eqn:symmetric_space_inequality}
H_t(a, b)^2H_t(b, c)^2 \leq H_t(a, c)H_t(a, s_b(c))H_t(a, a)^2
\end{equation}

To see this, substitute $0$ for $a$, $x$ for $b$, and $x + y$ for $c$. Then $s_b(c)$ is $x - y$, and we recover the original inequality.

It is shown Appendix \ref{app:symmetric_space_diffusion} that any symmetric space which satisfies \eqref{eqn:symmetric_space_inequality} has monotonic diffusion. This inequality doesn't hold for arbitrary symmetric spaces; it is shown in Appendix \ref{app:H3} that it fails on $\mathbb{H}^3$. However, we already know from \cite{RSD} that \eqref{eqn:symmetric_space_inequality} holds for flat tori, and the author has numerically tested \eqref{eqn:symmetric_space_inequality} for the 2-sphere and real projective 2-space, using the spherical harmonic functions of SciPy to approximate the heat kernel. The inequality appears to hold for these spaces. It seems natural then to ask: for which Riemannian symmetric spaces does \eqref{eqn:symmetric_space_inequality} hold? Perhaps it holds on all Riemannian symmetric spaces of compact type.

It can also be readily seen from equation (1.6) of \cite{GN} that the heat kernel has monotonic diffusion on $\mathbb{H}^3$, despite the fact that \eqref{eqn:symmetric_space_inequality} fails on this space. Perhaps monotonic diffusion holds on other hyperbolic spaces as well.

\section{Ackowledgements}

The train of thought that lead to the main idea of this paper began with a discussion with Alexander Borisov. He had noticed that strongly positive-definite functions obey the main inequality of \cite{RSD}. It was through trying to relate this to the fact that Gaussian pushforwards obey the same inequality that I came across the construction used in section 3. So I thank him for the stimulating discussion, as well as for helpful comments on a draft. I thank Oded Regev as well for helpful comments and discussion.

This paper was written while I was visiting the University of M\"unster. I thank Christopher Deninger for making this possible.

\appendix
\section{Appendix: Lack of monotonic diffusion on arbitrary manifolds} \label{app:monotonicity_counterexample}

We present here an argument of Jeff Cheeger that not all Riemannian manifolds have monotonic diffusion.

It is shown in \cite{CV} that, on a compact manifold $M$ of dimension at least $3$, any finite sequence $\lambda_1 \leq \lambda_2 \leq \cdots \leq \lambda_n$ of positive real numbers appears as the first nonzero eigenvalues, counted with multiplicity, of the Laplacian on $M$ for some choice of metric. In particular, there exist compact Riemannian manifolds on which the first nonzero eigenvalue of the Laplacian has multiplicity $1$. Let $M$ be such a manifold, with $\lambda_1$ the first eigenvalue and $\lambda_2$ the second. Let $f: M \to \mathbb{R}$ be an eigenfunction of the Laplacian corresponding to the eigenvalue $\lambda_1$. Choose points $x$ and $y$ in $M$ so that $f(y) > f(x) > 0$. We have the following, for some positive constant $C$:
$$H_t(x, x) = 1 + Ce^{-\lambda_1t}f(x) + O(e^{-\lambda_2t})$$
$$H_t(x, y) = 1 + Ce^{-\lambda_1t}f(y) + O(e^{-\lambda_2t})$$
Since $\lambda_2 > \lambda_1$, we have that $H_t(x, y) > H_t(x, x)$ for sufficiently large $t$. Since the quotient $H_t(x, y) / H_t(x, x)$ is eventually greater than $1$, but it converges to $1$ as $t \rightarrow \infty$, it can't be monotonically increasing. Therefore, we don't have monotonic diffusion on $M$.

\section{Appendix: Monotonic diffusion on Riemannian symmetric spaces} \label{app:symmetric_space_diffusion}

Here we show that, if the heat kernel of a Riemannian symmetric space $M$ satisfies the inequality
\begin{equation} \tag{2}
H_t(a, b)^2H_t(b, c)^2 \leq H_t(a, c)H_t(a, s_b(c))H_t(a, a)^2
\end{equation}
then it has monotonic diffusion.
First of all, note that \eqref{eqn:symmetric_space_inequality} implies the following analogue of  \ref{thm:chi_inequality_rollary}:

\begin{equation} \label{eqn:heat_kernel_lemma}
\frac{H_t(a, b)H_t(b, c)}{H_t(a, a)} \leq \frac{H_t(a, c) + H_t(a, s_b(c))}{2}
\end{equation}

This follows from dividing both sides of \eqref{eqn:symmetric_space_inequality} by $H_t(a, a)^2$, and then applying the AM-GM inequality. This is essentially the same as the proof of \ref{thm:chi_inequality_rollary}.

From this, we can derive the monotonic diffusion inequality in a way that parallels the proof of \ref{thm:chi_convolve_even}. We have, for $t' \geq t$:

$$H_{t'}(x, y) = \int_M H_{t' - t}(x, z)H_t(z, y)dz$$
$$ = \int_M H_{t' - t}(x, s_x(z))H_t(s_x(z), y)dz$$
$$ = \int_M H_{t' - t}(x, z)H_t(s_x(z), y)dz$$

Averaging the first and third lines above, we have:

$$H_{t'}(x, y) = \int_M H_{t' - t}(x, z)\frac{H_t(z, y) + H_t(s_x(z), y)}{2}dz$$
$$ = \int_M H_{t' - t}(x, z)\frac{H_t(y, z) + H_t(y, s_x(z))}{2}dz$$

Applying \eqref{eqn:heat_kernel_lemma} then gives

$$H_{t'}(x, y) \geq \int_M H_{t' - t}(x, z)\frac{H_t(y, x)H_t(x, z)}{H_t(y, y)}dz$$
$$= \int_M H_{t' - t}(x, z)H_t(z, x)dz\frac{H_t(x, y)}{H_t(y, y)}$$
$$= \frac{H_{t'}(x, x)H_t(x, y)}{H_t(y, y)}$$

Dividing by $H_{t'}(x, x)$ gives the result.

\section{Appendix: Failure of \eqref{eqn:symmetric_space_inequality} on $\mathbb{H}^3$} \label{app:H3}

Here we present an argument, due to Oded Regev, that the inequality \eqref{eqn:symmetric_space_inequality} fails on hyperbolic 3-space. In this section we realize $\mathbb{H}^3$ through the hyperboloid model, as described in \cite{CFKP}. In other words, we see $\mathbb{H}^3$ as the set of points $(x_0, x_1, x_2, x_3)$ in $\mathbb{R}^4$ with $x_0^2 - x_1^2 - x_2^2 - x_3^2 = 1$. The geodesic distance between two points $x = (x_0, x_1, x_2, x_3)$ and $y = (y_0, y_1, y_2, y_3)$ in $\mathbb{H}^3$ is then given by

$$d(x, y) = \arccosh(x_0y_0 - x_1y_1 - x_2y_2 - x_3y_3)$$

Suppose $d_1 \in \mathbb{R}_{>0}$. We define the following points in $\mathbb{H}^3$:

$$a = (\cosh(d_1), \sinh(d_1), 0, 0)$$
$$b = (1, 0, 0, 0)$$
$$c = (\cosh(d_1), 0, \sinh(d_1), 0)$$

Then we have $d_1 = d(a,b) = d(b,c)$, and we define $d_2$ as $d(a, c) = \arccosh(\cosh(d_1)^2)$. We also fix some arbitrary $t \in \mathbb{R}_{>0}$. We have a formula for the heat kernel on $\mathbb{H}^3$, given in \cite{GN}. When the geodesic distance between $x$ and $y$ is given by $d$, we have:

$$H_t(x, y) = \frac{1}{(4\pi t)^{3/2}}\frac{d}{\sinh(d)}\exp(-t-\frac{d^2}{4t})$$

Using this formula, the inequality \eqref{eqn:symmetric_space_inequality} reduces to the following, for our choices of $a$, $b$, $c$, and $t$:

$$\frac{d_1^2}{\sinh(d_1)^2}\exp(-\frac{d_1^2}{2t}) \leq \frac{d_2}{\sinh(d_2)}\exp(-\frac{d_2^2}{4t})$$

Let $LS$ and $RS$ be the left and right sides of the above inequality. We have, as $d_1$ and $d_2$ grow:

$$\ln(LS) = \frac{-d_1^2}{2t} + O(d_1)$$
$$\ln(RS) = \frac{-d_2^2}{4t} + O(d_2)$$

However, it can be readily seen that $d_2 = 2d_1 + O(1)$. We therefore have
$$\ln(RS) = \frac{-d_1^2}{t} + O(d_1),$$
and therefore $RS < LS$ for sufficiently large $d_1$. This violates \eqref{eqn:symmetric_space_inequality}.

\end{document}